\documentclass{article}
\usepackage{amssymb,amsmath,amsthm,natbib,amsthm,graphicx,anysize,epstopdf,hyperref,dsfont,pdfsync,comment,color}
\usepackage{geometry}
\hypersetup{colorlinks=false}
\newtheorem{theorem}{Theorem}
\newtheorem{corollary}{Corollary}

\newtheorem{lemma}{Lemma}
\numberwithin{equation}{section}
\begin{document}
\begin{center} \large  \sc Reducing bias in nonparametric density estimation via bandwidth dependent kernels: $L_{1}$ view\normalsize \rm \footnote{We thank an anonymous referee and an Associate Editor for excellent comments that improved this note significantly.}  \\[0.25in]
\begin{tabular}{l}
\multicolumn{1}{c}{ \sc Kairat Mynbaev}\\[.2in]
International School of Economics\\
Kazakh-British Technical University\\
Tolebi 59\\
Almaty 050000, Kazakhstan\\
email: kairat\_mynbayev@yahoo.com\\
\end{tabular}\\[.1in]
\begin{center}
and\\[.1in]
\end{center}
\begin{tabular}{lcl}
\multicolumn{3}{c}{ \sc Carlos Martins-Filho}\\[.2in]
Department of Economics &  &IFPRI \\
University of Colorado &  & 2033 K Street NW  \\
Boulder, CO 80309-0256, USA& \&&Washington, DC 20006-1002, USA\\
email: carlos.martins@colorado.edu& & email: c.martins-filho@cgiar.org\\
\end{tabular}\\[.2in]

November, 2016\\[.2in]
\end{center}
\noindent\bf Abstract. \rm  We define a new bandwidth-dependent kernel density estimator that improves existing convergence rates for the bias, and preserves that of the variation, when the error is measured in $L_1$.  No additional assumptions are imposed to the extant literature.  \\[.1in]

\noindent \bf AMS-MS Classification. \rm 62G07, 62G10, 62G20.
\clearpage
\setlength{\baselineskip}{24pt}
\pagestyle{plain}
\setcounter{page}{1}
\setcounter{footnote}{0}
\section{Introduction}
Given a sequence of $n \in \mathds{N}$ independent realizations $\{X_j\}_{j=1}^n$ of the random variable $X$, having density $f$ on $\mathds{R}$, the Rosenblatt-Parzen kernel estimator (\cite{Rosenblatt1956}, \cite{Parzen1962}) of $f$  is given by
\begin{equation}\label{first}
f_{n}(x)=\frac{1}{n}\sum_{j=1}^{n}(S_{h_n}K)(x-X_{j}),
\end{equation}
where $S_{h_n}$ is an operator defined by
\begin{equation}\label{1}
(S_{h_n}K)(x)=\frac{1}{h_n}K\left( \frac{x}{h_n}\right),
\end{equation}
$K$ is a kernel, i.e., a function on $\mathds{R}$ such that $\int K(x)dx=1$ and $h_n>0$ is a non-stochastic bandwidth such that $h_n \rightarrow 0$ as $n\rightarrow \infty$.\footnote{Throughout this note, integrals are over $\mathds{R}$, unless otherwise specified.}

One of the most natural and mathematically sound (\cite{Devroye1985}, \cite{Devroye1987}) criteria to measure the performance of $f_{n}$ as an estimator of $f$ is the $L_{1}$ distance $\int |f_{n}-f|$.  In particular, given that this distance is a random variable (measurable function of $\{X_j\}_{j=1}^n$) it is convenient to focus on $E \left( \int |f_{n}-f| \right)$, where $E$ denotes the expectation taken using $f$.  For this criterion, there is a simple bound \cite[p. 31]{Devroye1987}
\begin{equation*}
E\left( \int |f_{n}-f|\right) \leq  \int | (f \ast S_{h_n}K)-f|+E\left(  \int |f_{n}-f\ast S_{h_n}K| \right),
\end{equation*}
where for arbitrary $f,g \in L_1$, $(f \ast g)(x)=\int g(y)f(x-y)dy$ is the convolution of $f$ and $g$.  The term $\int |f \ast S_{h_n}K-f|$ is called bias over $\mathds{R}$ and $E \left( \int |f_{n}-f\ast S_{h_n}K| \right)$ is called the variation over $\mathds{R}$.  There exists a large literature devoted to establishing conditions on $f$ and $K$ that assure suitable rates of convergence of the bias to zero as $n \rightarrow \infty$ (see, \it inter alia\rm, \cite{Silverman1986}, \cite{Devroye1987} and \cite{Tsybakov2009}).  In particular, if $K$ is of order $s$, i.e., $\alpha _{j}(K)=0$ for $j=1,...,s-1$ and $\alpha _{s}(K)\neq 0$, where $\alpha _{j}(K)=\int t^{j}K(t)dt$ is the $j$th moment of $K$, and $f$ has an integrable derivative $f^{(s)}$, then $\int |f \ast S_{h_n}K-f|$ is of order $O(h_n^{s})$ and this order cannot be improved, see, e.g., \cite[Theorem 7.2]{Devroye1987}.  In this note, we show that if in \eqref{1} the kernel is allowed to depend on $n$, then the order $O(h_n^{s})$ can be replaced by the order $o(h_n^{s})$, without increasing the order of the kernel or the smoothness of the density.   In addition, another result from \cite{Devroye1987} states that if $K$ is a kernel of order greater than $s$ and the derivative $f^{(s)}$ is $a$-Lipschitz then the bias is of order $O(h_n^{s+a})$.  We achieve the same rate of convergence with kernels of order $s$.

\section{Main results}

Let $L_{1}$ and $C$ denote the spaces of integrable and (bounded) continuous functions on $\mathds{R}$ with norms $\left\Vert f\right\Vert _{1}=\int |f|$ and $\left\Vert f\right\Vert _{C}=\sup |f|$, and $\beta _{s}(K)=\int |t|^{s}\left\vert K(t)\right\vert dt$.  Let $\{K_n\}$ be a sequence of kernels and define
\begin{equation*}
\hat{f}_n(x) = \frac{1}{n}\sum_{j=1}^n(S_{h_n}K_n)(x-X_j).
\end{equation*}
In the following Theorem \ref{thm1}, the density $f$ has the same degree of smoothness and the kernels $K_n$ are of the same order as in \cite[Theorem 7.2]{Devroye1987}, but the bias is of order $o(h_n^{s})$ instead of $O(h_n^{s})$.  This results because the kernels depend on $n$ and have ``disappearing" moments of order $s$.

\begin{theorem}\label{thm1}
Let $\{K_n\}$ be a sequence of kernels of order $s$ such that: 1. $\alpha_s(K_n) \rightarrow 0$; 2. $\{u^sK_n(u)\}$ is uniformly integrable.  For all $f$ with absolutely continuous $f^{(s-1)}$ and $f^{(s)} \in L_1$, we have $\|f\ast S_{h_n}K_n-f\|_1=o(h_n^{s})$.
\end{theorem}

\begin{proof}
Note that since $K_n$ is a kernel
\begin{equation}\label{27}
f\ast S_{h_n}K_n(x)-f(x)=\int K_n(t)[f(x-h_nt)-f(x)]dt.
\end{equation}
Since $f$ is $s$-times differentiable, by Taylor's Theorem,
\begin{equation*}
f(x-h_nt)-f(x)=\sum_{j=1}^{s-1}\frac{f^{(j)}(x)}{j!}(-h_nt)^{j}+\int_{x}^{x-h_nt}\frac{
(x-h_nt-u)^{s-1}}{(s-1)!}f^{(s)}(u)du.
\end{equation*}
Furthermore, given that $K_n$ is of order $s$,
\begin{equation}\label{doisoito}
f\ast S_{h_n}K_n(x)-f(x)=\frac{1}{(s-1)!} \int \int_x^{x-h_nt}(x-h_nt-u)^{s-1}f^{(s)}(u)du K_n(t)dt.
\end{equation}
Letting $\lambda = -\frac{u-x}{h_nt}$ we have
\begin{equation}\label{doisnove}
\int_x^{x-h_nt}(x-h_nt-u)^{s-1}f^{(s)}(u)du=(-h_nt)^s \int_0^1 f^{(s)}(x-h_n\lambda t)(1-\lambda)^{s-1}d\lambda.
\end{equation}
Substituting \eqref{doisnove} into \eqref{doisoito} we obtain
\begin{equation}\label{210}
f\ast S_{h_n}K_n(x)-f(x)=\frac{(-h_n)^s}{s!}\int \int_0^1 f^{(s)}(x-h_n\lambda t)s(1-\lambda)^{s-1}d\lambda t^sK_n(t)dt.
\end{equation}
Since $\int_0^1(1-\lambda)^{s-1}d\lambda = \frac{1}{s}$, we have that
\begin{equation}\label{211}
\frac{(-h_n)^s}{(s-1)!}\int \int_0^1f^{(s)}(x)(1-\lambda)^{s-1}d\lambda t^sK_n(t)dt=\frac{(-h_n)^s}{s!}f^{(s)}(x)\int t^s K_n(t)dt.
\end{equation}
Then, adding and subtracting \eqref{211} to the right-hand side of \eqref{210} gives
\begin{align*}
f\ast S_{h_n}K_n(x)-f(x)&=\frac{(-h_n)^s}{s!} \left( f^{(s)}(x)\alpha_s(K_n)+\int \int_0^1 [f^{(s)}(x-h_n\lambda t)-f^{(s)}(x)]s(1-\lambda)^{s-1}d \lambda t^sK_n(t)dt \right) .
\end{align*}
Since $f^{(s)}\in L_{1}$ we write its continuity modulus as $\omega(\delta)=\sup_{|t|\leq \delta }\int\left\vert f^{(s)}(x-t)-f^{(s)}(x)\right\vert dx$.  It is well-known (see properties M.2, M.6 and M.7 in \cite{Zhuk2003}) that
\begin{equation} \label{212}
\omega (\delta)\leq 2\left\Vert f^{(s)}\right\Vert _{1},\ \omega \text{ is
nondecreasing and }\lim_{\delta \rightarrow 0}\omega (\delta )=0.
\end{equation}
Then,
\begin{align}
\|f\ast S_{h_n}K_n-f\|_1& \leq  \frac{h_n^s}{s!} \left[ \left\Vert f^{(s)}\right\Vert _{1}\left\vert
\alpha _{s}(K_n)\right\vert+\int\int_{0}^{1}\int\left\vert f^{(s)}(x-h_n\lambda t)-f^{(s)}(x)\right\vert dx\,s(1-\lambda)^{s-1}d\lambda| t^{s}K_n(t)| dt\right] \notag\\
&\leq  \frac{h_n^s}{s!} \left[ \left\Vert f^{(s)}\right\Vert _{1}\left\vert
\alpha _{s}(K_n)\right\vert+\int_{0}^{1}\int \omega(\lambda h_n |t|) \left\vert\ t^{s}K_n(t)\right\vert dt \, s(1-\lambda)^{s-1}d\lambda\right] \notag \\
&=  \frac{h_n^s}{s!} \left[ \left\Vert f^{(s)}\right\Vert _{1}\left\vert
\alpha _{s}(K_n)\right\vert+\int_{0}^{1}\left( \int_{|t|\leq \frac{1}{\sqrt{h_n}}} + \int_{|t|>\frac{1}{\sqrt{h_n}}}  \right)\omega(\lambda h_n |t|)  |t^{s}K_n(t)| dt \, s(1-\lambda)^{s-1}d\lambda\right]  \label{213} \\
& \leq \frac{h_n^s}{s!} \left[ \left\Vert f^{(s)}\right\Vert _{1}\left\vert
\alpha _{s}(K_n)\right\vert+\omega\left(\sqrt{h_n}\right) \beta_s(K_n)+ 2  \left\Vert f^{(s)}\right\Vert _{1}\int_{|t|>\frac{1}{\sqrt{h_n}}} \left\vert t^{s}K_n(t)\right\vert dt \right].\label{208}
\end{align}
Given that $\alpha_s(K_n)\rightarrow 0$ as $n \rightarrow \infty$, $\{t^s K_n(t)\}$ is uniformly integrable, which implies $\underset{n}{\text{sup}}\,\beta_s(K_n)<\infty$, and using \eqref{212} and \eqref{208} we have
\begin{equation}\label{209}
\|f\ast S_{h_n}K_n-f\|_1=o(h_n^s).
\end{equation}
\end{proof}
\noindent \bf Remark 1. \rm Kernel sequences $\{K_n\}$ that satisfy the restrictions imposed by Theorem \ref{thm1} can be easily constructed.  To this end, denote by $\mathcal{B}_{s}$ the space of functions with bounded norm $\left\Vert K\right\Vert _{\mathcal{B}_{s}}=\beta _{0}(K)+\beta _{s}(K)$.  Take functions $K_{(0)},K_{(s)}\in \mathcal{B}_{s}$ such that
\begin{equation}
\alpha _{0}(K_{(0)})=1,\, \alpha _{j}(K_{(0)})=0\text{ for }j=1,...,s;\,\alpha _{j}(K_{(s)})=0\text{ for }j=0,1,...,s-1,\ \alpha _{s}(K_{(s)})=1.
\label{214}
\end{equation}
We define the $n$-dependent kernel $K_n=K_{(0)}+h_nK_{(s)}$ with $0< h_n \leq 1$.  Note that $K_n$ is a kernel of order $s$ with $\alpha _{s}(K_n)=h_n$ which tends to zero as $n \rightarrow \infty$.  It is clear that any kernel $K$ of order $s$ can be written as $K=K_{(0)}+\alpha_s(K)K_{(s)}$, so that the conventional $s$-order kernels obtain from ours with $\alpha_s(K)=h_n$.  Furthermore, it follows from \eqref{214} that $\{t^s(K_{(0)}(t)+h_nK_{(s)}(t))\}$ is uniformly integrable.

Now, to obtain $K_{(0)}$ and $K_{(s)}$, assume that for a nonnegative kernel $K$ we have $\beta_{2s}(K)<\infty$.  Then, we can associate with $K$ a symmetric matrix
\begin{equation*}
A_{s}=\left(
\begin{array}{cccc}
\alpha _{0}(K) & \alpha _{1}(K) & ... & \alpha _{s}(K) \\
\alpha _{1}(K) & \alpha _{2}(K) & ... & \alpha _{s+1}(K) \\
... & ... & ... & ... \\
\alpha _{s}(K) & \alpha _{s+1}(K) & ... & \alpha _{2s}(K)
\end{array}
\right),
\end{equation*}
such that $\det A_s\neq 0$ (see \cite{Mynbaev2014}). For an arbitrary vector $b\in \mathds{R}^{s+1}$ let $a=A_s^{-1}b$ and define a polynomial transformation of $K$ by $(T_aK)(t)=\big(\sum_{i=0}^sa_it^i\big)K(t)$.  Then, we put $K_{(0)}=T_aK$ with $b=(1,0,...,0)'$ and $K_{(s)}=T_aK$ with $b=(0,...,0,1)'$, which satisfy \eqref{214}.  Thus, we have the following corollary to Theorem \ref{thm1}.
\begin{corollary}
Let $K_n=K_{(0)}+h_nK_{(s)}$ where $K_{(0)}$ and $K_{(s)}$ are as defined in Remark 1.  Then, for all $f$ with absolutely continuous $f^{(s-1)}$ and $f^{(s)} \in L_1$, we have $\|f\ast S_{h_n}K_n-f\|_1=o(h_n^{s})$.
\end{corollary}

\noindent \bf Remark 2. \rm If $K_{n}$ are supported on $[-M,M]$ for some $M>0$ and for all $n$, then in \eqref{213}, instead of splitting $\mathds{R}=\left\{ \left\vert t\right\vert \leq 1/\sqrt{h_n}\right\} \cup \left\{ \left\vert t\right\vert >1/\sqrt{h_n}\right\}$ we can use $\mathds{R}=\left\{ \left\vert t\right\vert \leq M\right\} \cup \left\{\left\vert t \right\vert >M\right\} $ and then instead of \eqref{208} we get
\begin{equation*}
\left\Vert f\ast S_{h_n}K_{n}-f\right\Vert _{1}\leq \frac{h_n^{s}}{s!}\left[\left\Vert f^{(m)}\right\Vert _{1}\left\vert \alpha _{s}(K_n)\right\vert+\omega (h_nM)\beta _{s}(K_n)\right] .
\end{equation*}
Hence, selecting $\left\{K_{n}\right\} $ in such a way that $\alpha_{s}(K_{n})=O(\omega (h_n))$, $\underset{n}{\sup} \,\beta _{s}(K_{n})<\infty$ and using the fact that $\omega (h_nM)\leq (M+1)\,\omega (h_n)$\footnote{See property M.5 in \cite{Zhuk2003}.} we get a result that is more precise than \eqref{209}, i.e.,
\begin{equation*}
\left\Vert f\ast S_{h_n}K_{n}-f\right\Vert _{1}=O(h_n^{s}\omega (h_n)).
\end{equation*}

\noindent \bf Remark 3. \rm By Young's inequality, the variation of $\hat{f}_n$ using $K_n=K_{(0)}+h_nK_{(s)}$ is such that
\begin{eqnarray*}
E\int |\hat{f}_{n}-f\ast S_{h_n}K_{n}|&\leq& E\int |\hat{f}_{n}-f\ast S_{h_n}K_{(0)}|+h_n\int |f\ast S_{h_n}K_{(s)}| \\
&\leq& E\int |\hat{f}_{n}-f\ast S_{h_n}K_{(0)}|+h_n\int |f|\int |K_{(s)}|.
\end{eqnarray*}
Letting $f_{n}^{(0)}$ be the estimator in \eqref{first} with $K=K_{(0)}$, we have $ E\int |\hat{f}_{n}-f\ast S_{h_n}K_{(0)}| \leq  E\int |f_{n}^{(0)}-f\ast S_{h_n}K_{(0)}| +h_n \int |f| \int|K_{(s)}|$.  Hence,
\begin{equation*}
E\int |\hat{f}_{n}-f\ast S_{h_n}K_{n}| \leq E\int |f_{n}^{(0)}-f\ast S_{h_n}K_{(0)}| + \,2\,h_n \int |f|\int|K_{(s)}|.
\end{equation*}
Since, $h_n \rightarrow 0$ as $n \rightarrow \infty$, the variation of $\hat{f}_{n}$ is asymptotically bounded by the variation of the conventional estimator $f_n$ using $K_{(0)}$, i.e., $E\int |f_{n}^{(0)}-f\ast S_{h_n}K_{(0)}|$.

Under the assumptions that $f$ has a variance, $\int (1+t^2)(K_{(0)}(t))^2dt<\infty$, \cite[Theorem 7.4]{Devroye1987} showed that $E\int |f^{(0)}_{n}-f\ast S_{h_n}K_{(0)}|=O( (nh_n)^{-1/2})$.  Thus,
\begin{equation*}
\sqrt{nh_n}E\int |\hat{f}_{n}-f\ast S_{h_n}K_{n}|= (1+(nh_n^3)^{1/2})O(1)=O(1) ,
\end{equation*}
where the last equality follows if $nh_n^3 \leq c<\infty$.

We now provide an analog for Theorem 7.1 in \cite{Devroye1987}. There, the bias order $O(h^{s+a})$ is achieved for kernels with orders greater than $s$, while in the following theorems we obtain the same order of bias for kernels of order $s$.

\begin{theorem}\label{nthm2}
Let $\{K_n\}$ be a sequence of kernels of order $s$ such that: 1. $\alpha_s(K_n) =O(h_n^a)$; 2. $\underset{n}{\sup} \,\beta_{s+a}(K_n)<\infty$, for some $a \in (0,1]$.  For $f$ with absolutely continuous $f^{(s-1)}$ and $f^{(s)} \in L_1$ assume that for some $0<c <\infty$
\begin{equation}\label{hc}
\omega(\delta)\leq c \,|\delta|^a.
\end{equation}
Then, $\|f\ast S_{h_n}K_n-f\|_1=O(h_n^{s+a})$.
\end{theorem}
\begin{proof}
As in the proof of Theorem \ref{thm1}, we have
\begin{align}
\|f\ast S_{h_n}K_n-f\|_1& \leq  \frac{h_n^s}{s!} \left[ \left\Vert f^{(s)}\right\Vert _{1}\left\vert
\alpha _{s}(K_n)\right\vert+\int_{0}^{1}\int \omega(\lambda h_n |t|) s(1-\lambda)^{s-1}\left\vert\ t^{s}K_n(t)\right\vert dt d\lambda \right]
\end{align}
By \eqref{hc}
\begin{equation*}
\|f\ast S_{h_n}K_n-f\|_1\leq \frac{h_n^{s}}{s!}\left[ \left\Vert f^{(s)}\right\Vert _{1}\left\vert \alpha _{s}(K_n)\right\vert
+ch_n^{a }\int \int_{0}^{1} s(1-\lambda)^{s-1}d\lambda \left\vert t^{s+a}K_n(t)\right\vert dt\right] .
\end{equation*}
Hence, under conditions 1. and 2. in the statement of the theorem we have $\|f\ast S_{h_n}K_n-f\|_1=O(h_n^{s+a})$.
\end{proof}

\noindent \bf Remark 4. \rm Practitioners may find condition \eqref{hc} too general, preferring more primitive conditions on $f^{(s)}$.  To this end, we say that a function $g$ defined on $\mathds{R}$ satisfies a \emph{global Lipschitz condition} of order $a \in (0,1]$ if there exist positive functions $l(x),r(x)$ such that
\begin{equation}\label{3}
|g(x-h)-g(x)|\leq l(x)|h|^{a }\text{ for }|h|\leq r\left( x\right),\, x\in \mathds{R}.
\end{equation}
The function $l$ is called a Lipschitz constant and the function $r$ is called a Lipschitz radius.  The class $Lip(a,\delta)$, for $\delta>1$, is defined as the set of functions $g$ which satisfy \eqref{3} with $l$ and $r$ such that
\begin{equation}\label{intcon}
\int (l(x)+r(x)^{-\delta})dx<\infty.
\end{equation}
In the next lemma we give two sufficient sets of conditions for $g\in Lip(a,\delta)$.  In the first case $g$ is compactly supported, and in the second it is not.

\begin{lemma}\label{lemma} a) Suppose $g$ has compact support, ${\rm supp}\, g\subseteq [-N,N]$ for some $N>0$, and $g$ satisfies the usual Lipschitz condition $|g(x-h)-g(x)|\le c|h|^a$ for any $x,h$ and some $a\in(0,1]$. Set $l(x)=c,\ r(x)=1$ for $|x|<N$ and $l(x)=0,\ r(x)=|x|-N$ for $|x|\ge N$. Then, $g\in Lip(a,\delta)$ with any $\delta>1$.

\noindent b) Suppose that $ |g^{(1)}(t)|\le ce^{-|t|},\ t\in \mathds{R}$. Let $l(x)=c\exp(-|x|/2+1/2),\ r(x)=(1+|x|)/2,\ x\in \mathds{R}$. Then, $g\in Lip(1,\delta)$ with any $\delta>1$.
\end{lemma}
\begin{proof}
a) If $|x|<N$, then $|g(x-h)-g(x)|\le|h|^al(x)$ for all $h$ (and not only for $|h|\le r(x)$). If $|x|\ge N$, then $|h|\le r(x)=|x|-N$ implies $|x-h|\ge |x|-|h|\ge N$, so that $|g(x-h)-g(x)|=0=|h|^al(x)$ for $|h|\le r(x)$.

b) Let $|x|\ge 1$. We have
\begin{equation}\label{bound}
|g(x-h)-g(x)|\le|h|\sup_{|t-x|\le |h|} |g^{(1)}(t)|\le|h|c\sup_{|t-x|\le |h|}e^{-|t|}.
\end{equation}
$|t-x|\le|h|\le r(x)=(1+|x|)/2$ implies $|t|=|x+t-x|\ge |x|-|t-x|\ge |x|/2-1/2\ge0$ and \eqref{bound} gives $|g(x-h)-g(x)|\le|h|c\exp(1/2-|x|/2)=|h|l(x)$ for $|h|\le r(x)$. Now let $|x|<1$. Then, $e^{-|x|/2}\ge e^{-1/2}$ so that by \eqref{bound}, $|g(x-h)-g(x)|\le|h|c\le|h|c\exp(1/2-|x|/2)=|h|l(x)$ for all $h$ and not only for $|h|\le r(x)$.  Condition \eqref{intcon} is obviously satisfied in both cases.
\end{proof}
By part a) of Lemma 1, compactly supported densities with derivative $f^{(s)}$ that satisfies the usual $a$-Lipschitz condition are such that $f^{(s)} \in Lip(a,\delta)$ for any $\delta >1$.  This corresponds to the case treated in Theorem 7.1 of \cite{Devroye1987}.  Part b) shows that for densities with unbounded domains, not covered by Theorem 7.1, if $f^{(s)}(x)$ has derivative that decays exponentially as $|x| \rightarrow \infty$, then $f^{(s)}\in Lip(1,\delta)$ for any $\delta >1$.  Next we provide a version of Theorem \ref{nthm2} for densities with derivative $f^{(s)} \in Lip(a,\delta)$. 

\begin{theorem}\label{nthm3}
Suppose that the density $f$ is such that its derivative $f^{(s)}$ belongs to $L_1$, $C$ (the respective norms are finite) and $Lip(a,\delta)$.  Let $\{K_n\}$ be a sequence of kernels of order $s$ such that: 1. $\alpha_s(K_n) =O(h_n^a)$; 2. $\underset{n}{\sup} \,\max\{\beta_{s+a}(K_n),\beta_{s+\delta}(K_n)\}<\infty$. Then, $\|f\ast S_{h_n}K_n-f\|_1=O(h_n^{s+a})$.
\end{theorem}
\begin{proof}
As in the proof of Theorem \ref{thm1}, we have
\begin{equation*}
\|f\ast S_{h_n}K_n-f\|_1 \leq  \frac{h_n^s}{s!} \left[ \left\Vert f^{(s)}\right\Vert _{1}\left\vert
\alpha _{s}(K_n)\right\vert+\int_{0}^{1}\int\int\left\vert f^{(s)}(x-h_n\lambda t)-f^{(s)}(x)\right\vert | t^{s}K_n(t)| dt\,  dx\, s(1-\lambda)^{s-1}d\lambda\right].
\end{equation*}
Let $\mathcal{I}(x)=\int\left\vert f^{(s)}(x-h_n\lambda t)-f^{(s)}(x)\right\vert | t^{s}K_n(t)|  dt $ and note that since $f^{(s)}\in Lip(a,\delta)$
\begin{align*}
\mathcal{I}(x)&=\left( \underset{\lambda h_n|t| \leq r(x)}{\int} + \underset{\lambda h_n|t| > r(x)}{\int} \right)\left\vert f^{(s)}(x-h_n\lambda t)-f^{(s)}(x)\right\vert | t^{s}K_n(t)|  dt\\
& \leq \underset{\lambda h_n|t| \leq r(x)}{\int} l(x) \lambda^a h_n^a |t|^{a+s}|K_n(t)|dt + \underset{\lambda h_n|t|>r(x)}{\int} |f^{(s)}(x-h_n\lambda t)-f^{(s)}(x)| |t^{s}K_n(t)|  dt \\
&\leq h_n^a \beta_{s+a}(K_n)l(x)+  \underset{\lambda h_n|t|>r(x)}{\int} |f^{(s)}(x-h_n\lambda t)-f^{(s)}(x)| |t^{s}K_n(t)|  dt .
\end{align*}
Letting $\mathcal{I}_1(x)=  \underset{\lambda h_n|t|>r(x)}{\int} |f^{(s)}(x-h_n\lambda t)-f^{(s)}(x)| |t^{s}K_n(t)|  dt$ we have
\begin{align*}
\mathcal{I}_1(x)&\leq \underset{\lambda h_n|t| > r(x)}{\int} \frac{1}{|t|^\delta} \left(|f^{(s)}(x-h_n\lambda t)|+|f^{(s)}(x)|\right) |t|^{s+\delta}|K_n(t)|  dt.
\end{align*}
Noting that $|t|^{-\delta}<\lambda^\delta h_n^\delta r(x)^{-\delta}$ and given that $\|f^{(s)}\|_C<\infty$, we obtain
$
\mathcal{I}_1(x)\leq  2 \|f^{(s)}\|_C\,\frac{\lambda^\delta h_n^\delta}{r(x)^\delta}\beta_{s+\delta}(K_n)
$.  Consequently,
\begin{equation}
\mathcal{I}(x) \leq   h_n^a \max\{\beta_{s+a}(K_n),\beta_{s+\delta}(K_n)\} \left( l(x)+  2 h_n^{\delta-a} \|f^{(s)}\|_C\,\frac{1}{r(x)^\delta} \right).
\end{equation}
Since $\int_0^1s(1-\lambda)^{s-1}ds=\frac{1}{s}$ and given \eqref{intcon}
\begin{equation}
\|f\ast S_{h_n}K_n-f\|_1 \leq  \frac{h_n^s}{s!} \left[ \left\Vert f^{(s)}\right\Vert _{1}\left\vert
\alpha _{s}(K_n)\right\vert+ \frac{1}{s} h_n^a \max\{\beta_{s+a}(K_n),\beta_{s+\delta}(K_n)\} \int  \left( l(x)+  2 \|f^{(s)}\|_C\,\frac{1}{r(x)^\delta} \right)dx\right].
\end{equation}
Thus, using conditions 1. and 2. in the statement of the theorem, we have $\|f\ast S_{h_n}K_n-f\|_1=O(h_n^{s+a})$.
\end{proof}

\noindent \bf Remark 5. \rm As in the case of Theorem \ref{thm1}, Theorems \ref{nthm2} and \ref{nthm3} do not address the construction of the kernel sequence $\{K_n\}$.  The following corollary to Theorem \ref{nthm3} shows that $K_n=K_{(0)}+h_n^a K_{(s)}$ is a suitable kernel sequence, where $K_{(0)}$ and $K_{(s)}$ are as defined above.

\begin{corollary}\label{cor2}
Suppose the density $f$ is such that its derivative $f^{(s)}$ belongs to $L_1$, $C$ and to $Lip(a,\delta)$, where $a\in(0,1],\ \delta>1$.  Let $K_{(0)},K_{(s)}$ satisfy \eqref{214} and belong to the intersection $\mathcal{B}_{s+a}\cap \mathcal{B}_{s+\delta}$.  Put $K_{n}=K_{(0)}+h_n^{a }K_{(s)},$ $0\leq h_n\leq 1.$ Then $K_{n}$
is a kernel of order $s$ for $h_n>0$ and $ \|f\ast
S_{h_n}K_{n}-f\|_1=O(h_n^{s+a }).$
\end{corollary}
\noindent The condition $K_{(0)} \in \mathcal{B}_{s+a }$ and the definition $K_{n}=K_{(0)}+h_n^{a }K_{(s)}$ can be replaced by $K_{(0)} \in \mathcal{B}_{s+1}$ and $K_{n}=K_{(0)}+h_nK_{(s)}$, respectively, without affecting the conclusion.

\setlength{\baselineskip}{12pt}
\bibliographystyle{elsart-harv.bst}
\bibliography{mynbaev_martins-filho_L1(2016)v4.bbl}
\end{document}